\documentclass[reqno]{amsart}
\newcommand \datum {June 23, 2024}
\usepackage{amssymb,latexsym}
\usepackage{amsmath}
\usepackage{hyperref}
\usepackage{url}
\usepackage{graphicx}
\usepackage[dvipsnames]{xcolor}
\usepackage{enumerate}
\usepackage{enumitem}
\usepackage{tikz, color}
\numberwithin{equation}{section}
\theoremstyle{plain}
 \newtheorem{theorem}{Theorem}
 \newtheorem{lemma}{Lemma}
 \newtheorem{corollary}{Corollary}
\newtheorem{fact}{Fact}
\theoremstyle{definition}
 \newtheorem{definition}{Definition}
\theoremstyle{remark}
\newtheorem{remark}{Remark}

\newcommand\alg[1]{\mathbb#1}
\newcommand \dual[1]{{#1^{\textup{du}}}}
\newcommand \vdual[1]{{#1{}^{\textup{du}}}}
\newcommand \ddual[1]{\vdual{(\dual{#1})}}
\newcommand \tpbg {paired-bipolar-graphs}
\newcommand \problem{\textup{PBGP}}
\newcommand \scheme {\textup{PBGS}}
\newcommand \initcnt[2] {\textup{Cnt}_{\textup{init,}#2}[#1]}
\newcommand \termcnt[2] {\textup{Cnt}_{\textup{term,}#2}[#1]}
\newcommand \btranspcnt[2] {\textup{Cnt}_{\textup{transp,}#2}[#1]}
\newcommand \edgecnt[3] {\textup{EfEdge}[#1,#2,#3]}
\newcommand \esetcnt[3] {\textup{EfSet}[#1,#2,#3]}
\newcommand \Sub [1] {\textup{Sub}(#1)}
\newcommand \Var [1] {\textup{Vrb}(#1)}
\newcommand \SSC{\textup{SSC}}
\newcommand \vvec[1] {\vec#1\kern1.5pt'}
\newcommand \transp [1] {#1^{\textup T}}
\newcommand \dtransp [1] { (#1^{\textup T})^{\textup T}}

\newcommand \alphref[1] {\ref{#1}} %with the enumitem package
\newcommand \source[1] {\textup{source}(#1)}
\newcommand \sink[1] {\textup{sink}(#1)}
\newcommand \thehead {\textup{head}}
\newcommand \thetail {\textup{tail}}
\newcommand \head[1] {\textup{\thehead}(#1)}
\newcommand \tail[1] {\textup{\thetail}(#1)}
\newcommand \Bdn {B_{\textup{dn}}}
\newcommand \Bup {B_{\textup{up}}}
\newcommand \Sid{S_{\textup{id}}}
\newcommand \theinc {\textup{inc}}
\newcommand \theout {\textup{out}}
\newcommand \inc[1] {\theinc(#1)}
\newcommand \out[1] {\theout(#1)}
\newcommand \bnd[1] {\textup{bnd}(#1)}
\newcommand \lbnd[1] {\textup{bnd}_{\textup{lft}}(#1)}
\newcommand \rbnd[1] {\textup{bnd}_{\textup{rght}}(#1)}

\newcommand \ovl [1]{\overline{#1}}
\newcommand \ul[1] {\underline{#1}}
\renewcommand \phi{\varphi}
\newcommand \Nnul {{\mathbb N}_0}
\newcommand \Nplu {{\mathbb N^+}}

\newcommand{\tbf}{\textbf}% text bold font
\newcommand{\set}[1]{\{#1\}}% set 
\newcommand\figwidthcoeff{0.95}
\newcommand\DeleteThis[1]{}
\newcommand\MaybeElsewhere[1]{}
\newcommand \red[1]{{\textcolor{red}{#1}\color{black}}}

\begin{document}

\title[Duality for paired bipolar graphs and submodule lattices]
{Duality for  pairs of upward bipolar plane graphs and submodule lattices}

\author[G.\ Cz\'edli]{G\'abor Cz\'edli}
\email{czedli@math.u-szeged.hu}
\urladdr{http://www.math.u-szeged.hu/~czedli/}
\address{University of Szeged, Bolyai Institute. 
Szeged, Aradi v\'ertan\'uk tere 1, HUNGARY 6720}

\begin{abstract} 
Let $G$ and $H$ be acyclic, upward bipolarly oriented plane graphs with the same number $n$ of edges. While $G$ can symbolize a flow network, $H$ has only a controlling role. Let  $\phi$ and $\psi$ be bijections from $\{1$, \dots, $n\}$ to the edge set of $G$ and that of $H$, respectively; their role is to define, for each edge of $H$, the \emph{corresponding} edge of $G$. Let $b$ be an element of an Abelian group $\alg A$. An  $n$-tuple $(a_1$, $\dots$, $a_n)$ of elements of $\alg A$ is a \emph{solution} of the \emph{\tpbg{} problem} $P:=(G,H$, $\phi,\psi$, $\alg A, b)$ if 
whenever $a_i$ is the ``all-or-nothing-flow'' capacity of the edge $\phi(i)$ for $i=1$, \dots, $n$ and $\vec e$ is a maximal directed path of $H$,
then by fully exploiting the capacities of the edges corresponding to the edges of $\vec e$ and neglecting the rest of the edges of $G$, we have a flow process transporting $b$ from the source (vertex) of $G$ to the sink of $G$. 
Let $P':=(H',G'$, $\psi',\phi'$, $\alg A, b)$, 
where $H'$ and $G'$ are the ``two-outer-facet'' duals of $H$ and $G$, respectively, and $\psi'$ and $\phi'$ are defined naturally. We prove that $P$ and $P'$ have the same solutions.
This result implies George Hutchinson's self-duality theorem on submodule lattices.
\end{abstract}

\dedicatory{Dedicated to M\'arta Madocsai, my brother Gy\"orgy, and to the memory of Mrs.\ P\'aln\'e Haraszti (born Anna Misk\'o), members of an old summer team}

\thanks{This research was supported by the National Research, Development and Innovation Fund of Hungary, under funding scheme K 138892.  \hfill{\red{\tbf{\datum}}}}


% 06C05(1980–now)Modular lattices, Desarguesian lattices
%05C21(2010–now)Flows in graphs


\maketitle

\section{Introduction}%\label{sect:intro} 

We present and prove the main result in Sections \ref{sect:intrexample}---\ref{sect:thm}, intended to be readable for most mathematicians.  
Section \ref{sect:hutch}, an application of the preceding sections, presupposes a modest familiarity with some fundamental concepts from (universal) algebra and, mainly, from lattice theory.

Sections \ref{sect:intrexample}--\ref{sect:thm} prove a duality theorem, Theorem \ref{thm:main}, for some pairs of finite, oriented planar graphs. The first graph, $G$, is a flow network with edge capacities belonging to a fixed Abelian group. The other graph plays a controlling role: each of its maximal paths determines a set of edges of $G$ to be used at their full capacities while neglecting the rest of the edges; see Section \ref{sect:intrexample} for a preliminary illustration.

Section \ref{sect:hutch} applies Theorem \ref{thm:main} to give a new and elementary proof of George Hutchinson's self-duality theorem on identities that hold in submodule lattices; 
our approach is simpler (mainly conceptually simpler) than the earlier ones.

The aforementioned two parts of the paper are interdependent. The second part, Section \ref{sect:hutch}, is based upon the first part (Sections \ref{sect:intrexample}--\ref{sect:thm}), while the necessity for a suitable tool in the second part led to the creation of the first part.

\section{An introductory example}\label{sect:intrexample}
Before delving into the technicalities of Section \ref{sect:problem}, consider the following example. 

\begin{figure}[ht] 
\centerline{ \includegraphics[width=\figwidthcoeff\textwidth]{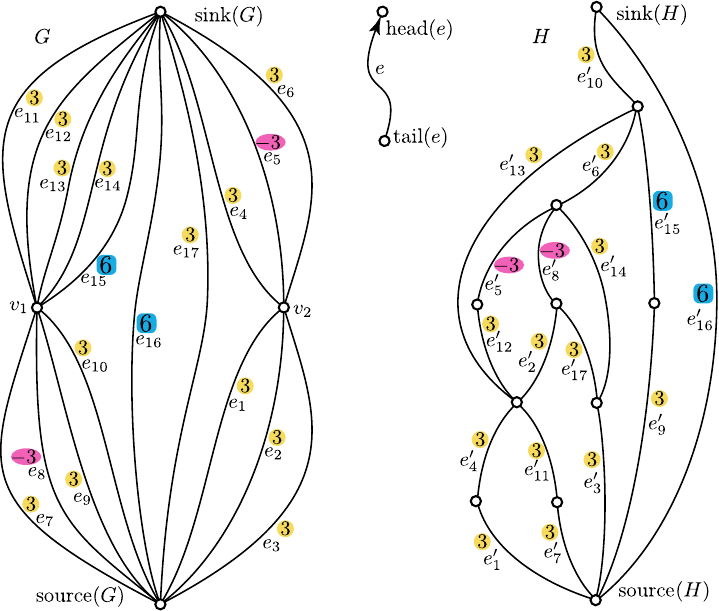}} \caption{An introductory example}\label{fig:egy}
\end{figure}

In Figure \ref{fig:egy}, $G$ and $H$ are oriented graphs. With the convention that \emph{every edge is upward oriented} (like in the case of Hasse diagrams of partially ordered sets), the arrowheads are omitted. The subscripts $1,\dots, 17$ supply a bijective correspondence between the edge set of $G$ and that of $H$. 
%At present, ignore the numbers in colored geometric shapes.
We can think of $G$ as a hypothetical \emph{concrete} system in which the arcs (i.e., the edges) are transit routes,  pipelines, fiber-optic cables, or freighters (or passenger vehicles) traveling on fixed routes, etc. 
The numbers in colored geometric shapes are the \emph{capacities} of the arcs of $G$. (Even though we repeat these numbers on the arcs of $H$, they still mean the capacities of the corresponding arcs of $G$; the arcs of $H$ have no capacities.)
These numbers are ``\emph{all-or-nothing-flow}'' capacities, that is, each arc should be either used at full capacity or avoided; this stipulation is due to physical limitations or economic inefficiency. (However, there can be parallel arcs with different capacities; see, for example, $e_{13}$ and $e_{15}$.) The vertices of $G$ are repositories (or warehouses, depots, etc.).
In contrast to $G$, the graph $H$ is to provide visual or digital information within a hypothetical control room.
Each maximal directed path of $H$ defines a method to transport exactly $6$ units  (such as pieces, tons, barrels, etc.) of something from $\source G$ to $\sink G$ without changing the final contents of other repositories. 
For example, $(e'_9$, $e'_{15}$, $e'_{10})$ is a maximal directed path of $H$; its meaning for $G$ is that we use exactly the arcs $e_9$, $e_{15}$, and $e_{10}$ of $G$. Namely, we use 
$e_9$ to transport $3$ (units of something) from $\source G$ to $v_1$, $e_{15}$ to transport 6 from $v_1$ to $\sink G$, and $e_{10}$ to transport 3 from $\source G$ to $v_1$. 
Depending on the physical realization of $G$, we can use $e_9$, $e_{15}$ and $e_{10}$ in this order, in any order, or simultaneously.
No matter which 
of the ten maximal paths of $H$ we choose, the result of the transportation is the same.
The negative sign of $-3$ at $e_5$ and $e_8$ means that the arc is to transport 3 in the opposite direction (that is, downward).  The scheme of transportation just described is very adaptive. Indeed, when choosing one of the ten maximal paths of $H$, several factors like speed, cost, the operational conditions of the edges, etc.\ can be taken into account.

\section{Paired-bipolar-graphs problems and schemes}\label{sect:problem}
First, we recall some, mostly well-known and easy, concepts and fix our notations. They are not unique in the literature, but we try to use the most expressive ones. We go mainly after
Auer at al.\ \cite{aueratal}\footnote{At the time of writing, freely available at 
\href{http://dx.doi.org/10.1016/j.tcs.2015.01.003}{http://dx.doi.org/10.1016/j.tcs.2015.01.003}  .}
and
Di Battista at al.\ \cite{dibattista}\footnote{At the time of writing, freely available at \href
{https://doi.org/10.1016/0925-7721(94)00014-X}
{https://doi.org/10.1016/0925-7721(94)00014-X} .}.
In the present paper, \emph{every} graph is assumed to be \emph{finite} and \emph{directed}. Sometimes, we say \emph{digraph} to emphasize that our graphs are directed. 
A (directed edge) $e$ of a graph starts at its \emph{tail}, denoted by  $\tail e$,  and ends at its \emph{head}, denoted by $\head e$. Occasionally, we say that $e$ \emph{goes} from $\tail e$ to $\head e$; see the middle of Figure \ref{fig:egy}. 
We can also say that $e$ is an \emph{outgoing edge} from $\tail e$ and an \emph{incoming edge} into $\head e$.  For a vertex $c$, let $\inc c$ and $\out c$ stand for the \emph{set of edges incoming into} $c$ and that of \emph{edges outgoing from $c$}, respectively. 
Sometimes, $\head e$ is denoted by an arrowhead put on $e$.
The \emph{vertex set} (set of all vertices)  and the \emph{edge set}  of a graph $G$ are denoted by $V(G)$ and $E(G)$, respectively. 
The graph containing no directed cycle is said to be  \emph{acyclic}. Such a graph has no loop edges, since there is no cycle of length 1, and it is \emph{oriented}, that is, $\inc{\tail e}\cap \out{\head e}=\emptyset$ for all $e\in E(G)$.
A vertex $c\in V(G)$ is a \emph{source} or a \emph{sink} if $\inc c=\emptyset$ or $\out c=\emptyset$, respectively.
A \emph{bipolarly oriented graph} or, briefly saying, a  \emph{bipolar graph} is an acyclic digraph that has exactly one source, has exactly one sink, and has at least two vertices. 
For such a graph $G$, $\source G$ and $\sink G$ denote the source and the sink of $G$, respectively. The uniqueness of $\source G$ and that of $\sink G$ imply that in a bipolar graph $G$,
\begin{equation}
\text{each maximal directed path goes from }\source G\text{ to }\sink G.
\label{eq:sGtosG}
\end{equation}
Next, guided by Section \ref{sect:intrexample} and Figure \ref{fig:egy},
we introduce the concept of a
\emph{\tpbg{} problem}.  
This problem with one of its solutions forms a \emph{\tpbg{} scheme}.  
For sets $X$ and $Y$, $X^Y$ denotes the set of functions from $Y$ to $X$.

\begin{definition}\label{def:solset} \ 
% Typesetter: there are two ways of the enumerate environement below; a wide one 
% and a narrow one.  The wide one begins with 
%\begin{enumerate}[label=\upshape(pb\arabic*), wide, labelwidth=!, labelindent=\parindent]
% The narrow one begins with \begin{enumerate}[label=\upshape(pb\arabic*)]

\begin{enumerate}[label=\upshape(pb\arabic*), wide, labelwidth=!, labelindent=\parindent]
\item%\label{def:solseta}
Assume that $G$ and $H$  are bipolar graphs with the same number $n$ of edges. Assume also that  $\phi\colon \set{1,\dots,n}\to E(G)$ and 
$\psi\colon \set{1,\dots,n}\to E(H)$ are bijections,  $e_i:=\phi(i)$ and $e'_i:=\psi(i)$ for $i\in\set{1,\dots,n}$; then $\phi\circ\psi^{-1}\colon E(H)\to E(G)$ defined by $e'_i\mapsto e_i$ is again a bijection. 
Let $\alg A=(A;+)$ be an Abelian group, and let $b$ be an element of $A$. (In Section \ref{sect:intrexample}, $\alg A=\alg Z$, the additive group of all integers, and $b=6$.)

\item\label{def:solsetb}
By a \emph{system of contents} we mean a function $S\colon V(G)\to A$, i.e., a member of $A^{V(G)}$. For $v\in V(G)$,   $S(v)\in A$ is the \emph{content} of $v$. 
The following three systems\footnote{The notations of these systems and other acronyms
 are easy to locate in the PDF of the paper. For example, in most PDF viewers, a search for ``Cntinit'' or ``bnd(''  gives the (first) occurrence of $\initcnt G b$ or $\bnd G$ (to be defined later), respectively.}
of contents deserve particular interest. 
The \emph{$b$-initial system of contents} is the function $\initcnt G b\colon V(G)\to A$ defined by 
\begin{equation*}
\initcnt G b(v)=
\begin{cases}
b, &\text{if }v=\source G,\cr
0=0_{\alg A}, &\text{if }v\in V(G)\setminus\set{\source G}.
\end{cases}
%\label{eq:initcnt}
\end{equation*}
The \emph{$b$-terminal system of contents} is $\termcnt G b\colon V(G)\to A$ defined by 
\begin{equation*}
\termcnt G b(v)=
\begin{cases}
b, &\text{if }v=\sink G,\cr
0, &\text{if }v\in V(G)\setminus\set{\sink G}.
\end{cases}
%\label{eq:termcnt}
\end{equation*}
The \emph{$b$-transporting system of contents} is $\btranspcnt G b\colon V(G)\to A$ defined by
\begin{equation}
\btranspcnt G b(v)=
\begin{cases}
-b, &\text{if }v=\source G,\cr
b, &\text{if }v=\sink G,\cr
0, &\text{if }v\in V(G)\setminus\set{\source G,\sink G}.
\end{cases}
\label{eq:btranspcnt}
\end{equation}

\item%\label{def:solsec} 
With respect to the pointwise addition, the systems of contents form an Abelian group, namely, a direct power of $\alg A$. The computation rule in this group is that  $(S^{(1)}\pm S^{(2)})(u)=S^{(1)}(u)\pm S^{(2)}(u)$ for all $u\in V(G)$. For example, $\termcnt G b=\initcnt G b + \btranspcnt G b$.

\item%\label{def:solsetd}
Let $\vec a:=(a_1,\dots,a_n)\in A^n$ be an $n$-tuple  of elements of $A$. The \emph{effect of an edge $e'_j$ of $H$ on $G$ with respect to $\vec a$} is the system $\edgecnt G{\vec a}{e'_j}$ of contents 
 defined by
\begin{equation}\edgecnt G{\vec a}{e'_j}(u):=
\begin{cases}
-a_j,&\text{if }u=\tail{e_j},\cr
a_j,&\text{if }u=\head{e_j},\cr
0,&\text{if }u\in V(G)\setminus\set{\tail{e_j},\head{e_j}};
\end{cases}
\label{eq:lRkSprdNnt}
\end{equation}
note that  $e'_j\in E(H)$ occurs on the left but $e_j\in E(G)$ on the right.

\item%\label{def:solsete}
For $\vec a:=(a_1,\dots,a_n)\in A^n$ and a directed path $\vvec e:=(e'_{j_1},e'_{j_2},\dots,e'_{j_k})$ in $H$ or a $k$-element subset $X=\set{e'_{j_1},e'_{j_2},\dots,e'_{j_k}}$ of $E(H)$, the 
\emph{effect of  $\vvec e$ or $X$ on $G$ with respect to $\vec a$}
is the following system of contents:
\begin{equation}
\esetcnt G{\vec a}{\set{e'_{j_1},\dots,e'_{j_k}}} := \sum_{i=1}^k 
\edgecnt G{\vec a}{e'_{j_i}}.
\label{eq:hpKpkRsmnRh}
\end{equation}

\item%\label{def:solsetf}
The \emph{\tpbg{} problem} is the 6-tuple  $(G,H$,  $\phi,\psi$, $\alg A, b)$, which we denote by 
\begin{equation}
\problem(G,H,\, \phi,\psi,\, \alg A, b).
\label{eq:MCDTP}
\end{equation}
We say that $\vec a:=(a_1,\dots,a_n)\in A^n$   is a \emph{solution} of this \tpbg{} problem 
if for each maximal directed path $\vvec e:=(e'_{j_1},e'_{j_2},\dots,e'_{j_k})$ in $H$,  
\begin{equation}
\esetcnt G{\vec a}{\set{e'_{j_1},\dots,e'_{j_k}}} =
\btranspcnt G b.
\label{eq:wWrsKHz}
\end{equation}

\item%\label{def:solsetg}
If $\vec a:=(a_1,\dots,a_n)\in A^n$ is a solution of $\problem(G,H$,  $\phi,\psi$, $\alg A, b)$, then we say that the 7-tuple 
 $(G,H$,  $\phi,\psi$, $\alg A, b,\vec a)$
is \emph{\tpbg{} scheme} and we denote this scheme by 
\begin{equation}
\scheme(G,H,\, \phi,\psi,\, \alg A, b,\vec a).
\label{eq:PBGS}
\end{equation}
\end{enumerate}
\end{definition}

For example, Figure \ref{fig:egy} determines $\problem(G,H,$ $\phi,\psi,$ $\alg Z, 6)$, where $\alg Z$ is the additive group of integers. As the numbers in colored geometric shapes form a solution, the figure defines a \tpbg{} scheme, too.  Even though we do not use the following two properties of the figure, we mention them. First, the \tpbg{} problem defined by the figure  has exactly one solution. Second,   if $k\in\Nplu:=\set{1,2,3,\dots}$,  we change $\alg Z$ to the $(2k)$-element additive group of integers modulo $2k$, and $b$ is (the residue class of) $1$ rather than 6,  then the problem determined by the figure has no solution.

The next section says more about \tpbg{} problems but only for specific bipolar graphs, including those in Figure \ref{fig:egy}.

\section{Bipolar plane graphs and the main theorem}\label{sect:thm}
For digraphs $G_1$ and $G_2$, a pair  $(\gamma,\chi)$ of functions is an \emph{isomorphism} from $G_1$ onto $G_2$ if 
both $\gamma\colon V(G_1)\to V(G_2)$ and $\chi\colon E(G_1)\to E(G_2)$ are bijections, and
for $e\in E(G_1)$, $\gamma(\tail e)=\tail{\chi(e)}$ and $\gamma(\head e)=\head{\chi(e)}$.
Hence,  $V(G_i)$ and $E(G_i)$ have been abstract sets and, in essence, a graph $G_i$ has been the system $(V(G_i),E(G_i), \thetail,\thehead)$ so far. 
However, in case of a \emph{plane graph} $G$, 
$V(G)$ is a finite subset of the plane $\mathbb R^2$ and $E(G)$ consists of oriented Jordan arcs (i.e., homeomorphic images of $[0, 1]\subseteq \mathbb R)$ such that each arc $e\in E(G)$ goes from a vertex $\tail e\in V(G)$ to a vertex $\head e\in V(G)$; see the middle part of Figure \ref{fig:egy}. On the other hand, $G$ is a \emph{planar graph} if it is isomorphic to a plane graph. Note the difference: a plane graph is always a planar graph but not conversely.

The \emph{boundary} $\bnd G$ of a plane graph consists of those arcs of $G$ that can be reached (i.e., each of their points can be reached) from any sufficiently distant point of the plane by walking along an open Jordan curve crossing no arc of the graph. Usually, we cannot define the boundary of a planar (rather than plane) graph.

\begin{definition}\label{def:upwBPG}
An \emph{upward bipolar plane graph}\footnote{\label{foot:widenupward}To widen the scope of the main result, our definition of ``upward'' is seemingly more general than the standard one occurring in the literature. However, up to graph isomorphism, our definition is equivalent to the standard one in which ``upward'' has its visual meaning; see Theorem \ref{thm:ascending}, taken from Platt \cite{platt},  later. Furthermore, if we went after the standard definition, then we should probably call the duals of these graphs ``rightward'', so we should introduce one more concept.} 
is  a bipolar plain graph $G$ such that both $\source G$ and $\sink G$ are on the boundary of $G$. An \emph{upward bipolarly oriented planar graph} is a digraph isomorphic to an upward bipolar plane graph.
\end{definition}

Next, let $G$ be an upward bipolar plane graph. The arcs of $G$ divide the plane into regions. Exactly one of these regions is geometrically unbounded; we call the rest of the regions \emph{inner facets}.  
Take a Jordan curve $C$ such\footnote{We stipulate that $C$ has exactly one point at infinity and, if possible, $C$ is a projective line.} that $C$ connects $\source G$ and $\sink G$ in the \emph{projective plane} and the affine part $C'$  (the set of those points of $C$ that are not on the line at infinity)  lies in the unbounded region. Then $C'$ divides the unbounded region into two parts called \emph{outer facets} . In Figure \ref{fig:ketto}, $C'$ is the union of the two thick dotted half-lines.  The \emph{facets} of $G$ are its inner facets and the two outer facets. In Figure \ref{fig:ketto}, any two facets of $G$ sharing an arc are indicated by different colors (or by distinct shades in a grey-scale version).

\begin{figure}[ht] 
\centerline{ \includegraphics[width=\figwidthcoeff\textwidth]{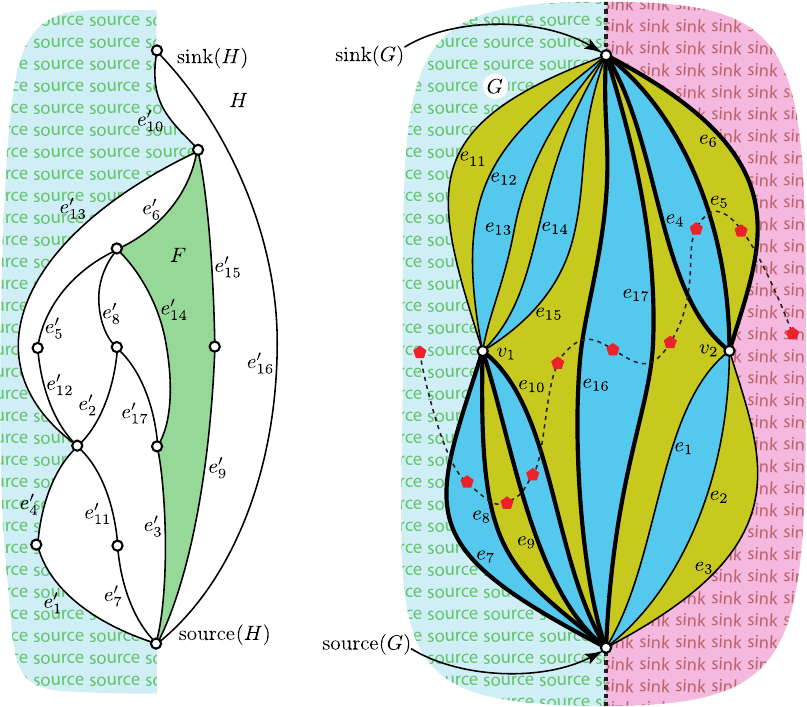}} \caption{A facet $F$ of $H$ and the facets of $G$}\label{fig:ketto}
\end{figure}

\begin{definition}\label{def:dualgraph}
For an upward bipolar plane graph $G$, we define the \emph{dual} of $G$, which we denote by $\dual G$, in the following way. Let $E(\dual G)$ be the set of all facets of $G$, including the two outer facets. For each edge $e\in G$, we define the \emph{dual edge} $\dual e$, as follows. Let $\tail{\dual e}$ and $\head{\dual e}$ be the two facets such that the arc $e$ is on their boundaries. Out of these two facets, $\tail{\dual e}$ is the one on the left when we walk along $e$ from $\tail e$ to $\head e$\footnote{\label{foot:lhrule}Miller and Naor \cite{millernaor} call this the ``left-hand rule'', since if our left thumb points in the direction of $e$, then the left index finger shows the direction of $\dual e$.}, while the other facet is $\head{\dual e}$. The edge set of $\dual G$ is $E(\dual G):=\{\dual e: e\in E(G)\}$. In Figure \ref{fig:ketto}, $\source {\dual G}$ and $\sink {\dual G}$ are the left outer facet and the right outer facet. (Only a bounded part of each of these two geometrically unbounded facets is drawn.) 
Note that $C'$ occurring before this definition belongs neither to $E(G)$ nor to $E{\dual G}$.
\end{definition}

Note that there are isomorphic upward bipolar \emph{plane} graphs $G_1$ and $G_2$ such that $\dual{G_1}$ and $\dual{G_2}$ are non-isomorphic; this is why we cannot define the dual of an upward bipolarly oriented \emph{planar} graph. 
Observe that the dual of an upward bipolar plane graph is a bipolar graph\footnote{This is why Definition \ref{def:dualgraph} deviates from the literature, where $\dual G$ has only one outer facet, the outer region. Fact \ref{fact:Pldual}, to be formulated later, asserts more.}, so the following definition makes sense.

\begin{definition}\label{def:dualproblem} 
With upward bipolar plane graphs $G$ and $H$, let $P:=\problem(G,H$, $\phi,\psi$, $\alg A, b)$ 
be a \tpbg{} problem;  see  \eqref{eq:MCDTP}. Define the bijections $\dual\psi \colon \set{1,\dots,n}\to E(\dual H)$ 
and $\dual\phi \colon \set{1,\dots,n}\to E(\dual G)$ by $\dual\psi(i):= \vdual {e'_i}=\dual{\psi(i)}$ and 
$\dual\phi(j):= \dual {e_j}=\dual{\phi(j)}$, respectively, for $i,j\in\set{1,\dots,n}$; here $\dual{e_j}$ and $\vdual{e'_i}$ are edges of the dual graphs defined in Definition \ref{def:dualgraph}. 
Then the \emph{dual of the \tpbg{} problem}  $P$ is the \tpbg{} problem
\begin{equation}
\dual P:=\problem(\dual H,\dual G,\,\dual\psi, \dual\phi,\, \alg A, b).
\label{eq:dualproblem}
\end{equation}
Briefly and roughly saying, we obtain the dual problem by interchanging the two graphs and dualizing both.
\end{definition}

Next, based on \eqref{eq:MCDTP}, \eqref{eq:PBGS}, and \eqref{eq:dualproblem}, we state our main theorem and a corollary.

\begin{theorem}\label{thm:main}
Let $P:=\problem(G,H$, $\phi,\psi$, $\alg A, b)$ be a \tpbg{} problem such that both $G$ and $H$ are upward bipolar plane graphs. Then $P$ and the dual problem $\dual P$ have exactly the same solutions.
\end{theorem}

This theorem, to be proved soon, trivially implies the following statement.

\begin{corollary}%\label{corol:PBGS}
For $G$, $H$, $\phi$, $\psi$, $\alg A$, and $b$ as in Theorem \ref{thm:main} and for every $\vec a$, $\scheme(G,H$, $\phi,\psi$, $\alg A, b,\vec a)$ is a \tpbg{} scheme if and only if so is
$\scheme(\dual H,\dual G$, $\dual \psi,\dual\phi$, $\alg A, b,\vec a)$.  
\end{corollary}

An arc $\{(x(t),y(t)): 0\leq t\leq 1\}$ in the plane is \emph{strictly ascending} if $y(t_1)<y(t_2)$ for all $0\leq t_1<t_2\leq 1$.  A plane graph is \emph{ascending} if all its arcs are strictly ascending. Platt \cite{platt} proved\footnote{Indeed, as $\source G$ and $\sink G$ are on $\bnd G$, we can connect them by a new arc without violating planarity. Furthermore, we can add parallel arcs to any arc. Thus, Platt's result applies.}  the following result, mentioned also in Auer at al.\ \cite{aueratal}.

\begin{theorem}[Platt \cite{platt}]\label{thm:ascending} 
Each upward bipolar plain graph is isomorphic to an upward bipolar ascending plain graph.
\end{theorem}

\begin{proof}[Proof of Theorem \ref{thm:main}]
Let $P$ and $\dual P$ be as in the theorem.
Theorem \ref{thm:ascending}  allows us to assume that $G$ and $H$ are upward bipolar ascending plain graphs; see Figure \ref{fig:egy} for an illustration. 
As the graphs are ascending, Figure \ref{fig:egy} satisfactorily reflects generality. Note that the summation in \eqref{eq:hpKpkRsmnRh} does not depend on the order in which the edges of a directed path are listed. Hence, we often give a directed path by the set of its edges. 
We claim that for any nonempty $X\subseteq \set{1,2,\dots,n}$,
\begin{equation}
\sum_{v\in V(G)}\esetcnt G{\vec a}{\set{e'_i:i\in X}}(v) =0.
\label{eq:szdhKrjPb}
\end{equation} 
For $|X|=1$, this is clear by  \eqref{eq:lRkSprdNnt}. 
The $|X|=1$ case and  \eqref{eq:hpKpkRsmnRh} imply the general case of \eqref{eq:szdhKrjPb}, since
\begin{align*}
\sum_{v\in V(G)}\esetcnt G{\vec a}{\set{e'_i:i\in X}}(v)= \sum_{v\in V(G)}\sum_{i\in X}\esetcnt G{\vec a}{\set{e'_i}}(v),
\end{align*} 
and the two summations after the equality sign above can be interchanged.

Assume that $\vec a\in A^n$ is a solution of 
$P$.  To show that $\vec a$ is a solution of $\dual P$, too,  take a maximal directed path $\Gamma=\set{\dual {e_i}:i\in M}$ in $\dual G$.
In Figure \ref{fig:ketto}, $M=\{7,8,9,10$, $16,17$, $4,5,6\}$ and, furthermore, $\set{e_i: i\in M}$ consists of the thick edges of $G$.
Note that  \eqref{eq:sGtosG},  with $H$ instead of $G$, is valid for $\Gamma$. 
Denote by $V(\Gamma)$ the set of vertices of path $\Gamma$; it consists of some facets of $G$. To mark these facets in the figure and also for a later purpose, for each facet $X\in V(\Gamma)$, we pick a point called \emph{capital}\footnote{Since we think of the facets as path-connected countries on the map.} in the geometric interior of $X$.  
These capitals are the red pentagon-shaped points in Figure \ref{fig:ketto}. We assume that the capital of $\source{\dual G}$, the left outer facet, is far on the left, that is, its abscissa is smaller than that of every vertex of $G$. Similarly, the capital of $\sink{\dual G}$ is far on the right. 
We need to show that
\begin{equation}
C:=\esetcnt {\dual H}{\vec a}{\set{\dual {e_i}:i\
\in M}}\text{ and } D:=\btranspcnt {\dual H} b 
\label{eq:CDmMknRlh}
\end{equation}
are the same.
So we need to show that for all   $F\in V(\dual H)$,  $C(F)=D(F)$. 

First, we deal with the case when $F$ is an internal facet of $H$; see on the left of  Figure \ref{fig:ketto}. As $H$ is ascending, the set of arcs on the boundary $\bnd F$ of $F$ is partitioned into a left half $\lbnd F$  and a right half $\rbnd F$. Furthermore, all arcs on $\bnd F$ (as well as in $H$) are ascending. 
Let $L:=\{i: e'_i$ belongs to $\lbnd F\}$ and 
$R:=\{i: e'_i$ belongs to $\rbnd F\}$. In Figure \ref{fig:ketto},  $L=\set{3,14,6}$ and $R=\set{9,15}$. For a directed path $\vec g$, let $\tail {\vec g}$ and $\head{\vec g}$ denote the tail of the first edge and the head of the last edge of $\vec g$, respectively.

For later reference, we point out that this paragraph to prove (the forthcoming)  \eqref{eq:wszBmnT} uses only the following property of $L$ and $R$: the directed paths 
\begin{equation}
\set{e'_i: i\in L}\text{ and }\set{e'_i: i\in R}\text{ have the same tail and the same head.}
\label{eq:dpsTlsmHd}
\end{equation}
Take a subset $K\subseteq \set{1,\dots,n}$ such that $K\cap L=\emptyset$ and  $\set{e'_i:i\in K\cup L}$ is a maximal directed path in $H$. 
In Figure \ref{fig:ketto}, $K=\set{10}$. Note that $K\cap R=\emptyset$ and $\set{e'_i:i\in K\cup R}$ is also a maximal directed path in $H$.
As $\vec a$ is a solution of $\problem(G,H$,  $\phi,\psi$, $\alg A, b)$, 
\begin{equation}
\esetcnt G{\vec a}{\set{e'_i:i\in L\cup K}}=
\esetcnt G{\vec a}{\set{e'_i:i\in R\cup K}},
\label{eq:wRgmWhBrgF}
\end{equation}
simply because both are $\btranspcnt G b$. 
By \eqref{eq:hpKpkRsmnRh}, both sides of \eqref{eq:wRgmWhBrgF} are sums. Subtracting 
$\esetcnt G{\vec a}{\set{e'_i:i\in K}}$ from both sides, we obtain that 
\begin{equation}
\esetcnt G{\vec a}{\set{e'_i:i\in L}}=
\esetcnt G{\vec a}{\set{e'_i:i\in R}}.
\label{eq:wszBmnT}
\end{equation}

Connect the capitals of the facets belonging to $V(\Gamma)$ by an open Jordan curve $J$ such that for each $\dual e\in E(\dual G)\setminus \Gamma$, the arc $e\in E(G)$ and $J$ have no geometric point in common and, furthermore, for each $\dual e\in \Gamma$, $J$ and $e$ has exactly one geometric point in common and this point is neither $\tail e$ nor $\head e$. In Figure \ref{fig:ketto}, $J$ is the thin dashed curve.
Let $\Bdn:=\{v\in V(G): v$ is (geometrically) below $J\}$. Similarly, let $\Bup$ be the set of those vertices of $G$ that are above $J$. Note that $\Bdn\cup\Bup=V(G)$ and $\Bdn\cap\Bup=\emptyset$. In Figure \ref{fig:ketto},  $\Bdn=\set{\source G, v_2}$ and $\Bup=\set{\sink G, v_1}$. 
Consider the sum
\begin{align}
\sum_{v\in \Bup}
\esetcnt G{\vec a}{\set{e'_i:i\in L}}(v)
&=\sum_{v\in \Bup}\sum_{i\in L}\edgecnt G{\vec a}{e'_i}(v)\label{eq:sZgkZlLbka}
\\
&=\sum_{i\in L}\sum_{v\in \Bup}\edgecnt G{\vec a}{e'_i}(v),
\label{eq:sZgkZlLbkb}
\end{align}
where the first equality comes from \eqref{eq:hpKpkRsmnRh}.
If $\{\tail{e_i}$, $\head{e_i}\}\subseteq\Bup$, then  
\begin{equation*}
\edgecnt G{\vec a}{e'_i}(\tail{e_i})=-a_i
\text{ and }
\edgecnt G{\vec a}{e'_i}(\head{e_i})=a_i,
%\label{eq:rkZbjSzkRbn}
\end{equation*}
in virtue of \eqref{eq:lRkSprdNnt},
eliminate each other in the inner summation in \eqref{eq:sZgkZlLbkb}. 
If  $\{\tail{e_i}$, $\head{e_i}\}\subseteq\Bdn$,
then $e'_i$ does not influence the inner summation at all. As $G$ is ascending, the case $\tail{e_i}\in \Bup$ and $\head{e_i}\in \Bdn$ does not occur.  So \eqref{eq:sZgkZlLbkb} depends only on those $i$ for which
$\tail{e_i}\in\Bdn$ and $\head{e_i}\in \Bup$. However, by the definitions of $\dual G$,  $\Gamma$, $J$, $\Bup$, and $\Bdn$, these subscripts $i$ are exactly the members of $M$.
Thus, we can change $i\in L$ in \eqref{eq:sZgkZlLbkb} to $i\in L\cap M$.  
For such an $i$, only $\head{e_i}$ is in $\Bup$ and, by \eqref{eq:lRkSprdNnt}, only $a_i$ contributes to the inner summation in \eqref{eq:sZgkZlLbkb}. Therefore, we conclude that 
\begin{equation}
\sum_{v\in \Bup}
\esetcnt G{\vec a}{\set{e'_i:i\in L}}(v)
=\sum_{i\in L\cap M}a_i.
\label{eq:lKtKblrHznt}
\end{equation}
As $L$ and $R$ have played the same role so far, we also have that 
\begin{equation}
\sum_{v\in \Bup}
\esetcnt G{\vec a}{\set{e'_i:i\in R}}(v)
=\sum_{i\in R\cap M}a_i.
\label{eq:rjBkSmrTlfT}
\end{equation}
Therefore, combining \eqref{eq:wszBmnT}, \eqref{eq:lKtKblrHznt}, and  \eqref{eq:rjBkSmrTlfT}, 
we obtain that 
\begin{equation}
\sum_{i\in L\cap M}a_i = \sum_{i\in R\cap M}a_i .
\label{eq:mKkhkrkkTsK}
\end{equation}
For $i\in L$ and $j\in R$,
by the left-hand rule quoted in Footnote \ref{foot:lhrule},  
 $\head{\vdual{e'_i}}=F$ and $\tail{\vdual{e'_j}}=F$. So,  at the $\overset \ast =$ sign below, we can use \eqref{eq:lRkSprdNnt} and that $F$ is not the endpoint of any further edge of $\dual H$.
Using  \eqref{eq:hpKpkRsmnRh}, \eqref{eq:CDmMknRlh}, 
\eqref{eq:lKtKblrHznt}, and \eqref{eq:rjBkSmrTlfT}, too,  
\begin{align}
&C(F)=\sum_{i\in M}\edgecnt{\dual H}{\vec a}{\dual{e_i}}(F)\cr
&\overset{\ast}=\sum_{i\in L\cap M}\edgecnt{\dual H}{\vec a}{\dual{e_i}}(F) + \sum_{j\in R\cap M}\edgecnt{\dual H}{\vec a}{\dual{e_j}}(F)\cr
&=\sum_{i\in L\cap M}a_i + \sum_{j\in R\cap M}(-a_j).
\label{eq:mTksTmgVkG}
\end{align}
Combining \eqref{eq:btranspcnt},  \eqref{eq:CDmMknRlh}, 
 \eqref{eq:mTksTmgVkG}, and \eqref{eq:mKkhkrkkTsK},  $C(F)=0=D(F)$, as required.

Next, we deal with the case $F=\source{\dual H}$. So $F$ is the outer facet left to $H$; see Figure \ref{fig:ketto}. We modify the earlier argument as follows. 
Let $R:=\{i: e'_i$ is on the left boundary of $H\}$. In  Figure \ref{fig:ketto}, $R=\{1$, $4$, $13$, $10\}$. Now $\set{\vdual{e'_i}:i\in R}$ is the set of outgoing edges from $F$ in $\dual H$.  As $\set{e'_i:i\in R}$ is a maximal directed path in $H$, 
\begin{equation}
\esetcnt G{\vec a}{\set{e'_i:i\in R}}=
\btranspcnt G b .
\label{eq:wRgSjnTvF}
\end{equation}
Similarly to \eqref{eq:sZgkZlLbka}--\eqref{eq:sZgkZlLbkb}, we take the sum
\begin{equation}
\sum_{v\in \Bup}
\esetcnt G{\vec a}{\set{e'_i:i\in R}}(v)
=
\sum_{i\in R}\sum_{v\in \Bup}\edgecnt G{\vec a}{e'_i}(v) .
\label{eq:sszBjktGb}
\end{equation}
Like earlier,  the inner sum in \eqref{eq:sszBjktGb} is 0 unless $\tail {e_i}\in\Bdn$ and $\head{e_i}\in\Bup$, that is, unless $i\in M$. 
Thus, we can change the range of the outer sum in \eqref{eq:sszBjktGb} from $i \in R$  to $i\in R\cap M$; note that $R\cap M=\{4, 10\}$ in Figure \ref{fig:ketto}. For $i\in R\cap M$, the inner sum is $\edgecnt G{\vec a}{e'_i}(\head{e_i})=a_i$. Therefore, \eqref{eq:sszBjktGb} turns into
\begin{equation}
\sum_{v\in \Bup}
\esetcnt G{\vec a}{\set{e'_i:i\in R}}(v)
=\sum_{i\in R\cap M}a_i.
\label{eq:krGhpNbrZm}
\end{equation}
So \eqref{eq:krGhpNbrZm}, \eqref{eq:wRgSjnTvF},  \eqref{eq:btranspcnt},  $\sink G \in \Bup$, and $\source G \notin \Bup$ imply that 
\begin{equation}
\sum_{i\in R\cap M}a_i = \sum_{v\in \Bup} \btranspcnt G b (v) = b.
\label{eq:mKmkzSr}
\end{equation}  
Similarly to \eqref{eq:mTksTmgVkG}, but now there is no incoming edge into $F=\source{\dual H}$ and so ``the earlier $L$'' is $\emptyset$ and not needed, we have that 
\begin{align}
&C(F)=\sum_{i\in M}\edgecnt{\dual H}{\vec a}{\dual{e_i}}(F)\cr
&=\sum_{i\in R\cap M}\edgecnt{\dual H}{\vec a}{\dual{e_i}}(F) =\sum_{i\in R\cap M}(-a_i)=-\sum_{i\in R\cap M} a_i.
\label{eq:hMrvHlkG}
\end{align}
By \eqref{eq:hMrvHlkG} and \eqref{eq:mKmkzSr},  $C(F)=-b$. Since 
$D(F) =  \btranspcnt {\dual H}b(\source{\dual H})=-b$ by \eqref{eq:btranspcnt} and \eqref{eq:CDmMknRlh}, we obtain the required equality $C(F)=D(F)$.

The treatment for the remaining case  $F=\sink{\dual H}$ could be similar, but we present a shorter approach. 
By  \eqref{eq:btranspcnt},  \eqref{eq:CDmMknRlh}, and the dual of \eqref{eq:szdhKrjPb}, 
\begin{equation}
\sum_{F\in V(\dual H)}C(F)=0=\sum_{F\in V(\dual H)}D(F).
\label{eq:bZbfsmzjn}
\end{equation}
We already know that for each $F\in V(\dual H)$ except possibly for $F=\sink {\dual H}$, $C(F)$ on the left of  \eqref{eq:bZbfsmzjn} equals the corresponding summand $D(F)$ on the right.  This fact and  \eqref{eq:bZbfsmzjn} imply that   $C(\sink{\dual H}=D(\sink{\dual H})$, as required.

After settling all three cases, we have shown that $C$ and $D$ in \eqref{eq:CDmMknRlh} are the same. This proves that any solution $\vec a$ of $P$ is also a solution of $\dual P$.

To prove the converse, we need the following easy consequence of Platt \cite{platt}.

\begin{fact}[Platt \cite{platt}]\label{fact:Pldual} If $X$ is an upward bipolar plane graph, then its dual, $\dual X$, is isomorphic to an upward bipolar plain graph.
\end{fact}

We can extract Fact \ref{fact:Pldual} from Platt \cite{platt} as follows.
As earlier but now for each facet $F$ of $X$, pick a capital $c_F$ in the interior of $F$. For any two neighboring facets $F$ and $T$, connect $c_F$ and $c_T$ by a new arc through the common bordering arc of $F$ and $T$. The capitals and the new arcs form a plane graph $X'$ isomorphic to $\dual X$, in notation, $X'\cong \dual X$. As $X'$ is an upward bipolar plain graph by Definition \ref{def:upwBPG},  we obtain Fact \ref{fact:Pldual}.

Temporarily, we call the way to obtain $X'$ from $X$ above a \emph{prime construction}; the indefinite article is explained by the fact that the vertices and the arcs of $X'$ can be chosen in many ways in the plane. The \emph{transpose} $\transp X$ of a graph $X$ is obtained from $X$ by reversing all its edges.  For $e\in E(X)$, $\transp e$ stands for the  \emph{transpose} of $e$; note that  $\tail{\transp e}=\head e$, $\head{\transp e}=\tail e$, and $V(\transp X)=\{\transp e:i\in V(X)\}$.

Resuming the proof of Theorem \ref{thm:main}, Theorem \ref{thm:ascending} allows us to assume that $G$ and $H$ are ascending. Let $G'$ be a plane graph obtained from $G$ by a prime construction;  $G'$ is isomorphic to $\dual G$.  
In Figure \ref{fig:ketto}, only some vertices of $G'$ are indicated by red pentagons and only some of its arcs are drawn as segments of the thin dashed open Jordan curve, but the figure is still illustrative. 
To obtain a graph $G''$ isomorphic to $\ddual G$, we apply a prime construction to $G'$  so that the vertices of $G$ are the chosen capitals that form $V(G'')$  and, geometrically, the original arcs of $G$ are the chosen arcs of $G''$ connecting these capitals. By the left-hand rule quoted in Footnote \ref{foot:lhrule}, $G''$ is $\transp G$. Hence $\ddual G\cong\transp G$. Similarly, $\ddual H\cong\transp H$. 
Let us  define $\transp \phi\colon\set{1,\dots,n} \to E(\transp G)$ and 
$\transp \psi\colon\set{1,\dots,n} \to E(\transp H)$ in the natural way by $\transp \phi(i):=\transp{(\phi(i))}$ and 
$\transp \psi(i):=\transp{(\psi(i))}$. We claim that 
\begin{equation}
P\text{ and }\transp P:=\problem(\transp G,\transp H,\, \transp \phi,\transp \psi,\, \alg A, b)\text{ have the same solutions.}
\label{eq:stFbkZdhVsN}
\end{equation}
The reason is simple: to neutralize that the edges are reversed, a solution $\vec u$ of $P$ should be changed to $-\vec u$. However, the source and the sink are interchanged, and this results in a second change of the sign. So, a solution of $P$ is also a solution of $\transp P$. Similarly, a solution of $\transp P$ is a solution of $\dtransp P=P$, proving \eqref{eq:stFbkZdhVsN}.

Finally, let $\vec a$ be a solution of $\dual P$. Fact \ref{fact:Pldual} allows us to apply the already proven part of Theorem \ref{thm:main} to $\dual P$ instead of $P$, and we obtain that $\vec a$ is a solution of $\ddual P$. We have seen that $\ddual G\cong \transp G$ and $\ddual H\cong \transp H$. Apart from these isomorphisms,  $\ddual \phi$ and $\ddual \psi$  are $\transp \phi$ and $\transp \psi$, respectively. Thus,  $\ddual P$ and $\transp P$ have the same solutions.  Hence $\vec a$ is a solution of $\transp P$, and so \eqref{eq:stFbkZdhVsN} implies that $\vec a$ is a solution of $P$, completing the proof of Theorem \ref{thm:main}.
\end{proof}

\begin{remark}\label{rem:elMnrPln}
Apart from applying the result of Platt \cite{platt}, the proof above is self-contained. Even though Platt's result may seem intuitively clear, its rigorous proof is not easy at all. Since a trivial induction instead of relying on Platt \cite{platt} 
would suffice for the particular graphs occurring in the subsequent section, our aim to give an \emph{elementary proof} of Hutchinson's self-duality theorem is not in danger.
\end{remark}

\section{Hutchinson's self-duality theorem}\label{sect:hutch}
The paragraph on pages 272--273 in \cite{hutczg} gives a detailed account on the contribution of each of the two authors of \cite{hutczg}. 
In particular, the self-duality theorem, to be recalled soon, is due exclusively to George Hutchinson.  Thus, we call it \emph{Hutchinson's self-duality theorem}, and we reference Hutchinson \cite{hutczg} in connection with it. A similar strategy applies when citing his other exclusive results from  \cite{hutczg}.

The original proof of the self-duality theorem is deep. It relies on Hutchinson \cite{hutch}, which belongs mainly to the theory of abelian categories, on the fourteen-page-long  Section 2 of Hutchinson and Cz\'edli \cite{hutczg}, and on the nine-page-long Section 3 of Hutchinson \cite{hutczg}. A second proof given  by Cz\'edli and Tak\'ach \cite{czgtakach} avoids Hutchinson \cite{hutch} and abelian categories, but relying on the just-mentioned Sections 2 and 3, it is still complicated. No elementary proof of Hutchinson's self-duality theorem has previously been given; in light of Remark \ref{rem:elMnrPln}, we present such a proof here.

By a \emph{module} $M$ over a ring $R$ with $1$ we always mean a \emph{unital left module}, that is, $1m=m$ holds for all $m\in M$. The lattice of all submodules of $M$ is denoted by $\Sub M$. For  $X,Y\in\Sub M$, $X\leq Y$ and $X\wedge Y$ means $X\subseteq Y$ and $X\cap Y$, respectively, while $X\vee Y$ is the submodule generated by $X\cup Y$. A \emph{lattice term} is built from variables and the operation symbols $\vee$ and $\wedge$. For lattice terms $p$ and $q$, the string ``$p=q$'' is called a \emph{lattice identity}. For example, 
$x_1\wedge(x_2\vee x_3)=(x_1\wedge x_2)\vee (x_1\wedge x_3)$ is a lattice identity; in fact, 
it is one of the two (equivalent) distributive laws. To obtain the \emph{dual} of a lattice term, we interchange $\vee$ and $\wedge$ in it. For example, the dual of
\begin{align}
r &=\Bigl(x_1\vee \bigl(x_2\wedge  (x_3\vee x_4)\bigr ) \vee x_5\Bigr)
\wedge \Bigl(\bigl( (x_6\vee x_7)\wedge (x_8\vee x_9) \bigr)  \vee x_{10}\Bigr),
\label{eq:rxmpla} \\
\noalign{\noindent which will be needed in an example later,  is the lattice term}
\dual r &=\Bigl(x_1\wedge \bigl(x_2\vee  (x_3\wedge x_4)\bigr ) \wedge x_5\Bigr)
\vee \Bigl(\bigl( (x_6\wedge x_7)\vee (x_8\wedge x_9) \bigr)  \wedge x_{10}\Bigr).\label{eq:rxmplb}
\end{align}
The \emph{dual} of a lattice identity is obtained by dualizing the lattice terms on both sides of the equality sign. For example, the dual of the above-mentioned distributive law is $x_1\vee(x_2\wedge x_3)=(x_1\vee x_2)\wedge (x_1\vee x_3)$, the other distributive law.

Now we can state Hutchinson's self-duality theorem.

\begin{theorem}[Hutchinson {\cite[Theorem 7]{hutczg}}]\label{thm:hutch} 
Let $R$ be a ring with $1$, and let $\lambda$ be a lattice identity. Then $\lambda$ holds in $\Sub M$ for all unital modules $M$ over $R$ if and only if so does the dual of $\lambda$. 
\end{theorem}

Even the following corollary of this theorem is interesting.
For $m\in \Nnul:=\set{0,1,2,\dots}$, let $\mathcal A_m$ be the class of Abelian groups\footnote{We note but do not need that the $\mathcal A_m$s are exactly the varieties of Abelian groups.}
 satisfying the identity $x+\dots+x=0$ with $m$ summands on the left. 
In particular, $\mathcal A_0$ is the class of all Abelian groups.

\begin{corollary}[Hutchinson \cite{hutczg}]\label{cor:hutch} For $m\in\Nnul$ and any lattice identity $\lambda$, $\lambda$ holds in the subgroup lattices $\Sub {\alg A}$ of all $\alg A\in\mathcal A_m$ if and only if so does the dual of $\lambda$. 
\end{corollary}

By treating each $\alg A\in\mathcal A_m$ as a left unital module over the residue-class ring $\mathbb Z_m$ in the natural way,  Corollary \ref{cor:hutch} follows trivially from Theorem \ref{thm:hutch}. 

In the rest of the paper, we derive Theorem \ref{thm:hutch} from Theorem \ref{thm:main}.

\begin{proof}[Proof of Theorem \ref{thm:hutch}]
We can assume that $\lambda$ is of the form $p\leq q$ where $p$ and $q$ are lattice terms. Indeed, any identity of the form $p=q$ is equivalent to the conjunction of $p\leq q$ and $q\leq p$. Thus, from now on, by a \emph{lattice identity} we mean a universally quantified inequality of the form 
\begin{equation}
\lambda: \quad 
(\forall x_1)\dots(\forall x_k)\Bigl( p(x_1,\dots,x_k)\leq q(x_1,\dots, x_k) \Bigr).
\label{eq:pleqqform}
\end{equation}

The dual of $\lambda$, denoted by $\dual \lambda$, is $\dual q\leq \dual p$, where $\dual p$ and $\dual q$ are the duals of the terms $p$ and $q$, respectively. Let us call $\lambda$ in \eqref{eq:pleqqform} a \emph{$1$-balanced identity} if every variable that occurs in the identity occurs exactly once in $p$ and exactly once in $q$. 
For lattice identities $\lambda_1$ and $\lambda_2$, we say that $\lambda_1$ and $\lambda_2$ are \emph{equivalent} if for every lattice $L$, $\lambda_1$ holds in $L$ if and only if so does $\lambda_2$.
As the first major step in the proof, we show that
for each lattice identity $p\leq q$, 
\begin{equation}
p\leq q\text{ is equivalent to a 1-balanced lattice identity } p'\leq q'.
\label{eq:dnsgslztt}
\end{equation}
To prove \eqref{eq:dnsgslztt}, 
observe that the absorption law $y=y\vee (y\wedge x)$ allows us to assume that every variable occurring in $p\leq q$ occurs both in $p$ and $q$. 
Indeed, if $x_i$ occurs, say, only in $p$, then we can change $q$ to $q\vee (q\wedge x_i)$. Let $B$ be the set\footnote{\eqref{eq:pleqqform} allows variables only from $\set{x_i:i\in\Nplu}$, so $B$ is a set. As usual,  $\Nplu=\{1,2,3,\dots\}$.} of those lattice  identities $\lambda$ in \eqref{eq:pleqqform}  
for which \eqref{eq:dnsgslztt} fails but the set of variables occurring in $p$ is the same as the set of variables occurring in $q$.  We need to show that $B=\emptyset$. Suppose the contrary.
For an identity $\lambda: p\leq q$ belonging to $B$, let $\beta(\lambda)$ be the number of those variables that occur at least three times in $\lambda$ (that is, more than once in $p$ or $q$). The notation $\beta$ comes from ``badness''. Pick a member $\lambda: p\leq q$ of $B$ that minimizes $\beta(\lambda)$. As $\lambda\in B$, we know that $\beta(\lambda)>0$. 
Let $x_1,\dots, x_k$ be the set of variables of $\lambda$. As $\beta(\lambda)$ remains the same when we permute the variables, we can assume that $x_1$ occurs in $\lambda$  at least three times. Let $u$ and $v$ denote the number of occurrences of $x_1$ in $p$ and that in $q$, respectively; note that $u,v\in\Nplu:=\set{1,2,3,\dots}$ and $u+v=\beta(\lambda)\geq 3$. Clearly, there is a $(u+k-1)$-ary term $\ovl p(y_1$, \dots, $y_u$, $x_2$, \dots, $x_k)$ such that each of $y_1$, \dots, $y_u$ occurs in $\ovl p$ exactly once and $p(x_1,\dots, x_k)$ is of the form
\begin{equation*}
p(x_1,x_2,\dots,x_k)=\ovl p(x_1,\dots, x_1,x_2,\dots, x_k)=\ovl p(x_1,\dots, x_1,\vvec x)
\end{equation*}
where $x_1$ is listed $u$ times in $\ovl p$ and $\vvec x=(x_2,\dots,x_k)$. For example, if 
\begin{equation*}
p(x_1,\dots,x_4)=\bigl((x_1\vee x_2)\wedge (x_1\vee x_3)\bigr) \wedge
\bigl((x_2\vee x_4)\wedge (x_1\vee x_3)\bigr),
\end{equation*}
then we can let
\begin{equation*}
\ovl p(y_1,y_2,y_3,x_2,x_3,x_4):=\bigl((y_1\vee x_2)\wedge (y_2\vee x_3)\bigr) \wedge
\bigl((x_2\vee x_4)\wedge (y_3\vee x_3)\bigr).
\end{equation*}
Similarly, there is an $(v+k-1)$-ary term $\ovl q(z_1$, \dots, $z_v$, $x_2$, \dots, $x_k)$ such that each of $z_1$, \dots, $z_v$ occurs in $\ovl q$ exactly once and $q(x_1,\dots, x_k)$ is of the form
\begin{equation*}
q(x_1,x_2,\dots,x_k)=\ovl q(x_1,\dots, x_1,x_2,\dots, x_k) =\ovl q(x_1,\dots, x_1,\vvec x)
\end{equation*}
where $x_1$ is listed $v$ times in $\ovl q$ and $\vvec x$ is still $(x_2,\dots, x_k)$
Consider the $u$-by-$v$ matrix $W=(w_{i,j})_{u\times v}$ of new variables; it has $u$ rows and $v$ columns.
Let
\[\vec w:=(w_{1,1}, w_{1,2},\dots,w_{1,v}, \,
w_{2,1}, w_{2,2},\dots,w_{2,v},\,\dots,\,
w_{u,1}, w_{u,2},\dots,w_{u,v})  
\]
be the vector of variables formed from the elements of $W$. That is, to obtain $\vec w$, we have listed the entries of $W$ row-wise.
We define the $(uv+k-1)$-ary terms
\begin{align*}
p^\ast(\vec w,\vvec x) &:=\ovl p(\bigwedge_{j=1}^v w_{1,j},\, \dots, \bigwedge_{j=1}^v w_{u,j},\,\vvec x) \text{ and}\cr
q^\ast(\vec w,\vvec x) &:=\ovl p(\bigvee_{i=1}^u w_{i,1},\, \dots, \bigvee_{i=1}^u w_{i,v},\,\vvec x),
\end{align*}
and we let $\lambda^\ast: p^\ast(\vec w,\vvec x)\leq q^\ast(\vec w,\vvec x)$. As each of the $w_{i,j}$s occurs in each of $p^\ast$ and $q^\ast$ exactly once and the numbers of occurrences of $x_2,\dots, x_k$ did not change, $\beta(\lambda^\ast)=\beta(\lambda)-1$. So, by the choice of $\lambda$,  we know that $\lambda^\ast$ is outside $B$. Thus, $\lambda^\ast$ is equivalent to a 1-balanced lattice identity. 

Next, we prove that $\lambda^\ast$ is equivalent to $\lambda$. 
Assume that $\lambda^\ast$ holds in a lattice $L$. Letting all the $w_{i,j}$s equal $x_1$ and using the fact that the join and the meet are idempotent operations, it follows immediately that $\lambda$ also holds in $L$. 
Conversely, assume that $\lambda$ holds in $L$, and let the $w_{i,j}$s and $x_2,\dots,x_k$ denote arbitrary elements of $L$. Since the lattice terms and operations are order-preserving, we obtain that 
\allowdisplaybreaks{
\begin{align*}
p^\ast(\vec w,\vvec x)&=\ovl p(\bigwedge_{j=1}^v w_{1,j},\, \dots, \bigwedge_{j=1}^v w_{u,j},\,\vvec x)
\cr
&\leq
\ovl p(\bigvee_{i=1}^u  \bigwedge_{j=1}^v w_{i,j},\, \dots, \bigvee_{i=1}^u  \bigwedge_{j=1}^v w_{i,j},\,\vvec x)=
p(\bigvee_{i=1}^u  \bigwedge_{j=1}^v w_{i,j},\,\vvec x) \cr
&\leq q(\bigvee_{i=1}^u  \bigwedge_{j=1}^v w_{i,j},\,\vvec x)
=\ovl q(\bigvee_{i=1}^u  \bigwedge_{j=1}^v w_{i,j},\, \dots, \bigvee_{i=1}^u  \bigwedge_{j=1}^v w_{i,j},\,\vvec x)
\cr
&\leq  
\ovl q(\bigvee_{i=1}^u   w_{i,1},\, \dots, \bigvee_{i=1}^u  w_{i,v},\,\vvec x)=q^\ast(\vec w,\vvec x),
\end{align*}
}%
showing that $\lambda^\ast$ holds in $L$. 
So $\lambda$ is equivalent to $\lambda^\ast$. Hence, $\lambda$ is equivalent to a $1$-balanced identity, since so is $\lambda^\ast$.
This contradicts that  $\lambda \in B$ and proves \eqref{eq:dnsgslztt}.

Clearly, if $\lambda$ is equivalent to a 1-balanced lattice identity $\lambda^\ast$, then the dual of $\lambda$ is equivalent to the dual of $\lambda^\ast$, which is again a 1-balanced identity. Thus, it suffices to prove Theorem \ref{thm:hutch} only for 1-balanced identities. So,  in the rest of the paper,
\begin{equation}
\lambda: p(x_1,\dots,x_n)\leq q(x_1,\dots,x_n)\,\text{ (in short, }p\leq q\text{)\, is a }1\text{-balanced}
\label{eq:whslmdbh}
\end{equation}
lattice identity.

\begin{figure}[ht] 
\centerline{ \includegraphics[width=\figwidthcoeff\textwidth]{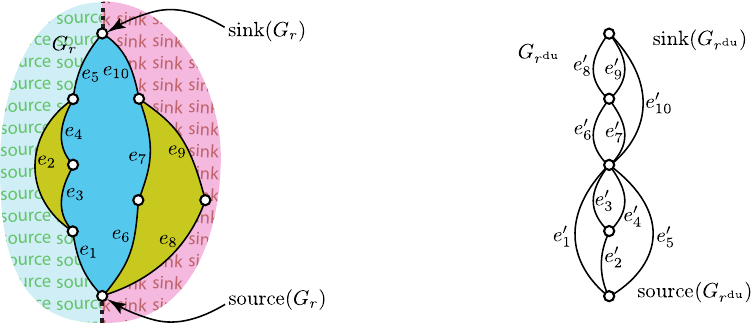}}
 \caption{For $r$ and $\dual r$ given in \eqref{eq:rxmpla} and \eqref{eq:rxmplb}, 
$G_r$ and its facets on the left,  and $G_{\dual r}$ on the right}
\label{fig:harom}
\end{figure}

For a lattice term $r$, $\Var r$ will stand for the \emph{set of variables} occurring in $r$.  We say that $r$ is \emph{repetition-free} if each of its variables occurs in $r$ only once, that is, if  $r\leq r$ is $1$-balanced. 
With the lattice terms given \eqref{eq:rxmpla} and \eqref{eq:rxmplb}, the following definition is illustrated by Figure \ref{fig:harom}.

\begin{definition}\label{def:sChlTrmk} 
 With each repetition-free lattice term $r$, we are going to associate an upward bipolar ascending plane graph $G_r$ up to isomorphism and a bijection $\xi_r\colon \Var r\to E(G_r)$ by induction as follows.
If $r$ is a variable, then $G_r$ is the two-element upward bipolar plane graph with a single directed edge, and $\xi_r$ is the only possible bijection from the singleton $\Var r$ to the singleton $E(G_r)$. 
For $r=r_1\vee r_2$, we obtain $G_r$ by putting $G_{r_2}$ atop ${G_{r_1}}$ and identifying (in other words, gluing together) $\sink {G_{r_1}}$ and $\source {G_{r_2}}$. Then $\source{G_r}=\source {G_{r_1}}$ and $\sink{G_r}=\sink{G_{r_2}}$.
For $r=r_1\wedge r_2$, we obtain  $G_r$ by bending or deforming, resizing, and moving  ${G_{r_1}}$ and ${G_{r_2}}$ so that
$\source {G_{r_1}}=\source {G_{r_2}}$, $\sink {G_{r_1}}=\sink {G_{r_2}}$, and the rest of ${G_{r_1}}$ is on the left of the rest of ${G_{r_2}}$. Then $\source{G_r}=\source {G_{r_1}}=\source {G_{r_2}}$ and $\sink{G_r}=\sink {G_{r_1}}=\sink {G_{r_2}}$. If $r=r_1\vee r_2$ or $r=r_1\wedge r_2$, then let $\xi_r:=\xi_{r_1}\cup \xi_{r_2}$, that is, for $i\in\set{1,2}$ and $x\in\Var{r_i}$,  $\xi_r(x):=\xi_{r_i}(x)$. 
\end{definition}
In the aspect of $G_r$, the lattice operations are associative but not commutative. A straightforward induction yields that for every lattice term $r$,
\begin{align}
&G_{\dual r}\text{ is isomorphic to }\dual {G_r}:=\dual {(G_r)}\text{, and} 
\label{eq:fcgdkfDla}\\
&\xi_{\dual r}(x)=\dual{\xi_r(x)}  \text{ for all } x\in\Var{\dual r};
\label{eq:fcgdkfDlb}
\end{align}
\eqref{eq:fcgdkfDla} and \eqref{eq:fcgdkfDlb} are exemplified by Figure \ref{fig:harom}, where $\dual{G_r}$ is given by facets.

The ring $R$ with 1 in the proof is fixed, and $(R,+)$ is its additive group. 
For  $p$ in \eqref{eq:whslmdbh}, 
we denote by  $\SSC(p,R)$ the \ul set of \ul systems of \ul contents of $G_p$ with respect to  $(R,+)$. That is, complying with the terminology of
Definition \ref{def:solset}\ref{def:solsetb},
\begin{equation}
\SSC(p,R)\text{ is }R^{V(G_p)},\text{ the set of all maps from }V(G_p)\text{ to }R.
\label{eq:SCntdf}
\end{equation}  
For a unital module $M$ over $R$ (an $R$-module $M$ for short), similarly to \eqref{eq:SCntdf}, let
\begin{equation*}
\SSC(p,M):=\Sub M^{V(G_p)},\text{ the set all } V(G_p)\to  M \text{ maps.}
\label{eq:modSCntd}
\end{equation*} 
Interrupting the proof of Theorem \ref{thm:hutch}, we formulate and prove two lemmas.

\begin{lemma}\label{lemma:grflm} For submodules $B_1$, \dots, $B_n$  and  elements $u,v$ of an $R$-module $M$, $v-u\in p(B_1,\dots,B_n)$ if and only if there exists an $S\in \SSC(p,M)$ such that 
\begin{equation}
S(\source{G_p})=u,\text{ }S(\sink{G_p})=v, \text{ and }
S(\head {e_i}) - S(\tail{e_i})\in B_i
\label{eq:mKgrMkpl}
\end{equation}
for  all edge $e_i\in E(G_p)$. The same holds with $q$ and $e'_i$ instead of $p$ and $e_i$, respectively.
\end{lemma}

Letting $u:=0$, the lemma describes the containment $v\in p(B_1,\dots,B_n)$. However, now that the lemma is formulated with $v-u$, it will be easier to apply it later. 
Based on the rule that for $B,B'\in \Sub M$, $B\vee B'=\set{h+h': h\in B,\,h'\in B'}$, we have that  $v-u\in B\vee B'$ if and only if there is a $w$ such that $w-u\in B$ and $v-w\in B'$. (For the ``only if'' part: $w:=u+h=v-h'$.) Hence, 
 the lemma follows by a trivial induction on the length of $p$; the details are omitted. 
Alternatively (but with more work), one can derive the lemma from the congruence-permutable particular case of  Cz\'edli \cite[Claim 1]{czgpuebla},  \cite[Proposition 3.1]{czgCPhorn},  \cite[Lemma 3.3]{czghorn} or  Cz\'edli and Day \cite[Proposition 3.1]{czgday} together with the canonical isomorphism between $\Sub M$ and the congruence lattice of $M$. The following lemma, in which $\xi_p$ and $\xi_q$ are defined in Definition \ref{def:sChlTrmk}, is crucial and less obvious.

\begin{lemma}\label{lemma:vZsv} Let $R$  be a ring with $1=1_R$ and let $\lambda: p\leq q$ be a $1$-balanced lattice identity as in  \eqref{eq:whslmdbh}. Then the following two conditions are equivalent.
%\begin{enumerate}[\kern3pt\upshape($\alpha$1)] %without the enumitem package
\begin{enumerate}[label=\upshape($\alpha$\arabic*)] % with the enumitem package
\item\label{lemma:vZsva} For every (unital left) $R$-module $M$, $p\leq q$ holds in $\Sub M$.
\item\label{lemma:vZsvb}
$\problem(G_p,G_q$, $\xi_p,\xi_q$, $(R,+), 1_R)$ has a solution. 
\end{enumerate}
\end{lemma}

\begin{proof}[Proof of Lemma \ref{lemma:vZsv}]
For $i\in\set{1,\dots,n}$, we denote $\xi_p(x_i)$ and $\xi_q(x_i)$ by $e_i$ and $e'_i$, respectively. 

Assume that \alphref{lemma:vZsva} holds.
Let $F$ be  the free unital $R$-module\footnote{We note but do not use the facts that $(R,+)$ can be treated an $R$-module denoted by $_RR$, and  $F$ that we are defining is the $|V(G_p)|$th direct power of $_RR$.} generated by  $V(G_p)$.
For each $e_i\in E(G_p)$, let $B_i\in \Sub F$ be the submodule generated by $\head{e_i}-\tail{e_i}$. In other words, $B_i=R\cdot(\head{e_i}-\tail{e_i}):=\set{r\cdot (\head{e_i}-\tail{e_i}):r\in R}$. 
Taking  $\Sid\in\SSC(p,F)$  defined by $\Sid(v):=v$ (like an identity map) for $v\in V(G_p)$, Lemma \ref{lemma:grflm} implies that $\sink{G_p}-\source{G_p}\in p(B_1,\dots, B_n)$.  So, as we have assumed \alphref{lemma:vZsva},  $\sink{G_p}-\source{G_p}\in q(B_1,\dots, B_n)$. Therefore, Lemma \ref{lemma:grflm} yields a system $T\in \SSC(q,F)$ of contents such that 
$T(\source{G_q})=\source{G_p}$, $T(\sink{G_q})=\sink{G_p}$, and for every $i\in\set{1,\dots,n}$, 
$
T(\head{e'_i})-T(\tail{e'_i}) \in B_i = R\cdot(\head{e_i}-\tail{e_i})
$.
Thus,  for each $i\in\set{1,\dots,n}$,  we can pick an $a_i\in R$ such that 
\begin{equation}
T(\head{e'_i})-T(\tail{e'_i}) = a_i\cdot (\head{e_i}-\tail{e_i}).
\label{eq:mRghmKsnGhn}
\end{equation}
Let $P$ stand for the \tpbg{} problem occurring in \alphref{lemma:vZsvb}.
With the $a_i$s in \eqref{eq:mRghmKsnGhn}, let $\vec a:=(a_1,\dots,a_n)$. We claim that $\vec a$ is a solution of $P$. To show this, let $\vvec e:=(e'_{j_1}$, \dots, $e'_{j_k})$ be a maximal directed path in $G_q$.
Let us compute, using the equality $\head{e'_{j_i}}=\tail{e'_{j_{i+1}}}$ for $i\in\set{1,\dots, k-1}$ at  $\overset{\ast\ast}=$  and \eqref{eq:mRghmKsnGhn} at $\overset\oplus=$:
\begin{align}
&\sink{G_p}-\source{G_p}=T(\sink{G_q})- T(\source{G_q}) \label{eq:wszVzhKlja}\\
&=  T(\head{e'_{j_k}})-T(\tail{e'_{j_1}})\overset{\ast\ast}=
\sum_{i=1}^k \bigl(T(\head{e'_{j_i}})-T(\tail{e'_{j_i}})\bigr)\cr
&\overset\oplus=\sum_{i=1}^k \bigl(a_{j_i}\cdot \head{e_{j_i}}-a_{j_i}\cdot \tail{e_{j_i}}\bigr).
\label{eq:wszVzhKljb}
\end{align}
For  $v\in V(G_p)$, define $I_v:=\{i: \tail{e_{j_i}}=v$ and $1\leq i\leq k\}$ and $J_v:=\{i: \head{e_{j_i}}=v$ and $1\leq i\leq k\}$. 
Expressing \eqref{eq:wszVzhKljb} as a linear combination of the free generators of $F$ with coefficients taken from $R$,  the coefficient of $v$ is  
$
\sum_{i\in J_v} a_{j_i} - \sum_{i\in I_v} a_{j_i}
$. Hence, it follows from \eqref{eq:lRkSprdNnt} and \eqref{eq:hpKpkRsmnRh} that 
\begin{align} 
\sum_{i\in J_v} a_{j_i} &- \sum_{i\in I_v} a_{j_i} 
=  
\sum_{i\in J_v} \edgecnt{G_p}{\vec a}{e'_{j_i}}(v)
- \sum_{i\in I_v} \edgecnt{G_p}{\vec a}{e'_{j_i}}(v)
\cr
&=\sum_{i\in\set{1,\dots,k}} \edgecnt{G_p}{\vec a}{e'_{j_i}}(v)
=
\esetcnt {G_p}{\vec a}{\set{e'_{j_1},\dots,e'_{j_k}}}(v)
\label{eq:smRznhtNGd}
\end{align}
is the coefficient of $v$ in  \eqref{eq:wszVzhKljb} and, by \eqref{eq:wszVzhKlja}, also in the linear combination expressing $\sink{G_p}-\source{G_p}$.
On the other hand, the coefficients of $\source {G_p}$, $\sink {G_p}$, and $v\in V(G_p)\setminus\{ \source {G_p}$, $\sink{G_p} \}$ in 
the straightforward linear combination expressing 
 $\sink{G_p}-\source{G_p}$ are  $-1_R$, $1_R$, and $0_R$, respectively. Since $F$ is freely generated by $V(G_p)$, this linear combination is unique. Therefore,  \eqref{eq:smRznhtNGd} is  $-1_R$, $1_R$, and $0_R$ for  $v=\source {G_p}$, $v=\sink {G_p}$, and $v\in V(G_p)\setminus\{ \source {G_p}$, $\sink{G_p} \}$, respectively. 
Thus, the function applied on the right of \eqref{eq:smRznhtNGd} to $v$ is the same as $\btranspcnt G {1_{\kern-1 pt R}}$ defined in \eqref{eq:btranspcnt}. As this holds for all $v\in V(G_p)$, the just-mentioned function equals $\btranspcnt G {1_{\kern-1 pt R}}$.  
Hence,  $\vec a$ is a solution of $P$; see  \eqref{eq:wWrsKHz}.  We have shown that \alphref{lemma:vZsva} implies \alphref{lemma:vZsvb}.

To show the converse implication, assume that  \alphref{lemma:vZsvb} holds, and let $\vec a$ be a solution of $P$. Let $M$ be an $R$-module, let $B_1,\dots, B_n\in\Sub M$, and let $v\in p(B_1,\dots,B_n)$. It is convenient to let $u=0_M$; then we obtain an $S\in\SSC(p,M)$ satisfying \eqref{eq:mKgrMkpl} for all $e_i\in V(G_p)$. Note in advance that when we reference Section \ref{sect:thm}, $\alg A:=(R,+)$, $G:=G_p$, and $H:=G_q$. 
For each $d\in V(G_q)$, 
\begin{equation}
\text{pick a directed path }\vvec e(d)=(e'_{j_1},\dots,e'_{j_k})\text{ from }\source{G_q}\text{ to }d;
\label{eq:msTkvgrnVBrm}
\end{equation}
here $k$ depends on the choice of this path (and on $d$). With reference to \eqref{eq:hpKpkRsmnRh}, let
\begin{equation}
T(d):=\sum_{w\in V(G_p)} 
\esetcnt {G_p}{\vec a}{\set{e'_{j_1},\dots,e'_{j_k}}}(w)\cdot S(w).
\label{eq:nLsKhgpZKs}
\end{equation}
We know from  Section \ref{sect:thm} that 
\eqref{eq:dpsTlsmHd} implies \eqref{eq:wszBmnT}. Hence, the coefficient of  $S(w)$ in \eqref{eq:nLsKhgpZKs} does not depend on the choice of $\vvec e(d)$. Thus, $T(d)$ is well defined, that is
\begin{equation}
T(d)\text{ does not depend on the choice of }\vvec e(d)\text{ in \eqref{eq:msTkvgrnVBrm}}.
\label{eq:rTghpQlBl}
\end{equation}
As $S(w)$ in \eqref{eq:nLsKhgpZKs} belongs to $M$ and its coefficient to $R$, $T(d)\in M$. So, $T\in \SSC(q,M)$. 
As the empty sum in $M$ is $0_M=u$, we have that $T(\source{G_q})=u$.  Since $\vec a$ is a solution of $P$ and $\vvec e(\sink{G_q})$ is a maximal directed path in $G_q$, it follows from  \eqref{eq:nLsKhgpZKs},  \eqref{eq:wWrsKHz}, \eqref{eq:btranspcnt}, 
and \eqref{eq:mKgrMkpl} that
\begin{equation*}
T(\sink{G_q})= 1_R\cdot S(\sink{G_p}) - 1_R\cdot S(\source{G_p})=v-u. 
%\label{eq:LszGknRflRfl}
\end{equation*}

To see the third part of \eqref{eq:mKgrMkpl}  with $q$ and $T$ instead of $p$ and $S$, let $e'_i\in E(G_q)$. According to \eqref{eq:msTkvgrnVBrm} but with $k-1$ instead of $k$, let  $\vvec e(\tail{e'_i})$ be the chosen directed path for $\tail{e'_i}\in V(G_q)$.
By \eqref{eq:rTghpQlBl}, we can assume that  $\vvec e(\head{e'_i})$ is obtained from $\vvec e(\tail{e'_i})$ by adding $e'_{j_k}:=e'_i$ to its end. So  $j_k=i$, $e'_{j_k}=e'_i$, 
\begin{equation*}
\vvec e(\tail{e'_i})=(e'_{j_1}, \dots, e'_{j_{k-1}}),\text{ and }
 \vvec e(\head{e'_i})=(e'_{j_1}, \dots,e'_{j_{k-1}}, e'_{j_{k}}). 
\end{equation*}
Hence, applying \eqref{eq:hpKpkRsmnRh}
to the coefficient of each of the $S(w)$ in  \eqref{eq:nLsKhgpZKs}, 
\begin{equation}
T(\head{e'_i})-T(\tail{e'_i})=
\sum_{w\in V(G_p)} 
\edgecnt {G_p}{\vec a}{e'_{j_k}}(w)\cdot S(w).
\label{eq:szrSgCskGk}
\end{equation}
As $j_k=i$ and most of the summands above are zero by \eqref{eq:lRkSprdNnt},  \eqref{eq:szrSgCskGk} turns into 
\begin{align*}
T(\head{e'_i})-T(\tail{e'_i})&=-a_i\cdot S(\tail{e_i}) + a_i\cdot S(\head{e_i})\cr 
&=a_i\cdot\bigl( S(\head{e_i})- S(\tail{e_i}) \bigr),
\end{align*}
which belongs to $B_i$ since $S$ satisfies \eqref{eq:mKgrMkpl}. Thus, Lemma \ref{lemma:grflm} yields that $v=v-u\in q(B_1,\dots,B_n)$. Therefore, $p(B_1,\dots,B_n)\leq q(B_1,\dots,B_n)$, that is, \alphref{lemma:vZsva} holds, completing the proof of Lemma \ref{lemma:vZsv}. 
\end{proof}

Next, we resume the proof of Theorem \ref{thm:hutch}. As noted in  \eqref{eq:whslmdbh}, $\lambda: p\leq q$ is 1-balanced. Clearly, so is $\dual \lambda: \dual q\leq \dual p$. 
Letting  $\mathcal L_R:=\{\Sub M: M$ is an $R$-module$\}$ and $P:=\problem(G_p,G_q$, $\xi_p,\xi_q$, $(R,+), 1_R)$,  
Lemma \ref{lemma:vZsv} gives that 
\begin{equation}
\lambda\text{ holds in }\mathcal L_r \iff P \text{ has a solution}.
\label{eq:bRlhwhdsRt}
\end{equation}
Tailoring Definition \ref{def:dualproblem} to the present situation, define $\dual{\xi_p} \colon\set{1,\dots,n}\to E(G_{\dual p})$ and 
 $\dual{\xi_q} \colon\{1$, \dots, $n\}\to E(G_{\dual q})$ in the natural way  by $\dual{\xi_p}(i):=\dual{\xi_p(i)}=\dual{e_i}$ and $\dual{\xi_q}(i):=\dual{\xi_q(i)}=\vdual{e'_i}$. With $P':=\problem(G_{\dual q}, G_{\dual p}$, $\dual{\xi_q},\dual {\xi_p}$, $(R,+), 1_R)$, Lemma \ref{lemma:vZsv} yields that 
\begin{equation}
\dual \lambda\text{ holds in }\mathcal L_r \iff P' \text{ has a solution}.
\label{eq:bsJssLzX}
\end{equation}
Let $\dual P$ denote the dual of $P$; see Definition \ref{def:dualproblem}. 
It follows from  \eqref{eq:fcgdkfDla}--\eqref{eq:fcgdkfDlb} and Definitions \ref{def:dualgraph} and \ref{def:dualproblem} that 
$P'$ is the same as $\dual P$. Hence, \eqref{eq:bsJssLzX} turns into
\begin{equation}
\dual \lambda\text{ holds in }\mathcal L_r \iff \dual P \text{ has a solution}.
\label{eq:bkMpSsKNgbd}
\end{equation}
Finally,  Theorem \ref{thm:main}, \eqref{eq:bRlhwhdsRt}, and \eqref{eq:bkMpSsKNgbd} imply that $\lambda$ holds in $\mathcal L_r$ if and only if so does $\dual \lambda$,  completing the proof of Theorem \ref{thm:hutch}.
\end{proof}

\end{document}